\documentclass[12pt]{amsart}
\usepackage{amscd}
\usepackage{graphicx}
\def\frk{\frak}               % font for "Fraktur"

\def\Phi{{\frk n}}
\def\Phi{{\frk N}}
\def\opn#1#2{\def#1{\operatorname{#2}}} % to make operators
\opn\chara{char} \opn\length{\ell} \opn\pd{pd} \opn\rk{rk}
\opn\projdim{proj\,dim} \opn\injdim{inj\,dim} \opn\rank{rank}
\opn\depth{depth} \opn\grade{grade} \opn\height{height}
\opn\embdim{emb\,dim} \opn\codim{codim}

\opn\Tr{Tr} \opn\bigrank{big\,rank}
\opn\superheight{superheight}\opn\lcm{lcm}
\opn\trdeg{tr\,deg}%\emph{
\opn\reg{reg} \opn\lreg{lreg} \opn\ini{in} \opn\lpd{lpd}
\opn\size{size}
\opn\div{div} \opn\Div{Div} \opn\cl{cl} \opn\Cl{Cl}
\opn\Spec{Spec} \opn\Supp{Supp} \opn\supp{supp} \opn\Sing{Sing}
\opn\Ass{Ass} \opn\Min{Min}
\opn\Ann{Ann} \opn\Rad{Rad} \opn\Soc{Soc}
\opn\Im{Im} \opn\Ker{Ker} \opn\Coker{Coker} \opn\Am{Am}
\opn\Hom{Hom} \opn\Tor{Tor} \opn\Ext{Ext} \opn\End{End}
\opn\Aut{Aut} \opn\id{id}

\opn\nat{nat}
\opn\pff{pf}%   \pf exists already
\opn\Pf{Pf} \opn\GL{GL} \opn\SL{SL} \opn\mod{mod} \opn\ord{ord}
\opn\Gin{Gin} \opn\Hilb{Hilb}
\opn\aff{aff} \opn\con{conv} \opn\relint{relint} \opn\st{st}
\opn\lk{lk} \opn\cn{cn} \opn\core{core} \opn\vol{vol}
\opn\link{link} \opn\star{star}
\opn\gr{gr}

\def\pot#1#2{#1[\kern-0.28ex[#2]\kern-0.28ex]}

\opn\dirlim{\underrightarrow{\lim}}
\opn\inivlim{\underleftarrow{\lim}}
\def\Implies{\ifmmode\Longrightarrow \else
        \unskip${}\Longrightarrow{}$\ignorespaces\fi}
\def\implies{\ifmmode\Rightarrow \else
        \unskip${}\Rightarrow{}$\ignorespaces\fi}
\def\iff{\ifmmode\Longleftrightarrow \else
        \unskip${}\Longleftrightarrow{}$\ignorespaces\fi}

\let\:=\colon
\newtheorem{Theorem}{Theorem}[section]
\newtheorem{Lemma}[Theorem]{Lemma}
\newtheorem{cor}[Theorem]{Corollary}
\newtheorem{Proposition}[Theorem]{Proposition}

\newtheorem{Example}[Theorem]{Example}

\newtheorem{Definition}[Theorem]{Definition}

\let\epsilon\varepsilon
\let\phi=\varphi
\let\kappa=\varkappa
\def\qed{\ifhmode\textqed\fi
      \ifmmode\ifinner\quad\qedsymbol\else\dispqed\fi\fi}
\def\textqed{\unskip\nobreak\penalty50
       \hskip2em\hbox{}\nobreak\hfil\qedsymbol
       \parfillskip=0pt \finalhyphendemerits=0}
\def\dispqed{\rlap{\qquad\qedsymbol}}

\opn\dis{dis}
\def\pnt{{\raise0.5mm\hbox{\large\bf.}}}

\opn\Lex{Lex}

\begin{document}
\title{Algebraic Characterization of the SSC $\Delta_s(\mathcal{G}_{n,r}^{1})$}

\author{Agha Kashif$^1$, Zahid Raza$^2$, Imran Anwar$^3$   }
 \thanks{ {\bf 1.} University of Management and Technology Lahore, Pakistan.\\
         {\bf 2.} University of Sharjah, College of Sciences, Department of Mathematics,United Arab Emirates.\\
         {\bf 3.} ASSMS, Government College University, Lahore, Pakistan.}
         %{\bf 3.} COMSATS Institute of Information Technology Lahore, Pakistan.}
\email {kashif.khan@umt.edu.pk, zraza@sharjah.ac.ae, iimrananwar@gmail.com}

 \maketitle
\begin{abstract}
In this paper, we characterize the set of spanning trees of
$\mathcal{G}_{n,r}^{1}$ (a simple connected graph consisting of $n$
edges, containing exactly one {\em $1$-edge-connected chain} of $r$
cycles $\mathbb{C}_r^1$ and
$\mathcal{G}_{n,r}^{1}\setminus\mathbb{C}_r^1$ is a {\em forest}).
We compute the Hilbert series of the face ring
$k[\Delta_s(\mathcal{G}_{n,r}^{1})]$ for the spanning simplicial
complex $\Delta_s(\mathcal{G}_{n,r}^{1})$. Also, we characterize
associated primes of the facet ideal
$I_{\mathcal{F}}(\Delta_s(\mathcal{G}_{n,r}^{1}))$. Furthermore, we
prove that the face ring $k[\Delta_s(\mathcal{G}_{n,r}^{1})]$ is
Cohen-Macaulay.
\end{abstract}
\section{introduction}
The study of simplicial complexes arising form a simple graph has
been an important topic and attracted good literature. One popular
chapter of this literature is the complementary simplicial complex
$\Delta_G$ of a graph $G$; for example, see \cite{Vi}. The notion of
spanning simplicial complex (SSC) $\Delta_s(G)$ associated to a
simple connected graph $G(V,E)$ was firstly introduced in
\cite{ARK}. For {\em uni-cyclic graphs} $U_{n,m}$, it is proved that
$\Delta_s(U_{n,m})$ is {\em shifted} in \cite{ARK}.  Zhu, Shi and
Geng \cite{Zh} further investigated the algebraic and combinatorial
properties of SSC associated to $n$-cyclic graphs with a common
edge. In \cite{KAR}, the authors investigated the algebraic
properties of SSC $\Delta_s(G_{n,r})$
 associated to {\em r-cyclic graphs} $G_{n,r}$ (containing exactly $r$ cycles having no edge in common). Moreover, they proved that the facet ideal $I_{\mathcal{F}}(\Delta_s(G_{n,r}))$ has linear quotients with respect to its generating set and computed the {\em betti numbers} of $I_{\mathcal{F}}(\Delta_s(G_{n,r}))$ for particular cases. Some other interesting classes of simple finite connected graphs are studied for SSC by Pan, Li and Zhu in \cite{PLZ} and Guo and Wu in \cite{GW}.\\
In this paper, we investigate the class of spanning simplicial
complexes $\Delta_s(\mathcal{G}_{n,r}^{1})$ associated to
$\mathcal{G}_{n,r}^{1}$. Where $\mathcal{G}_{n,r}^{1}$ is a
connected graph having $n$ edges, containing exactly one {\em
$1$-edge-connected chain} of $r$ cycles $\mathbb{C}_r^1$ and
$\mathcal{G}_{n,r}^{1}\setminus\mathbb{C}_r^1$ is a {\em forest}. In other words, $\mathcal{G}_{n,r}^{1}$ is a graph consisting of $r$ cycles such that every pair of consecutive cycles have exactly one edge common between them. If
$C_1,C_2,\ldots,C_r$ are the $r$ cycles of the graph
$\mathcal{G}_{n,r}^{1}$ forming $\mathbb{C}_r^1$ with respective
lengths $m_1,m_2,\ldots, m_r$ then we fix the label of edge set of
$\mathcal{G}_{n,r}^{1}$ as follows;
\begin{equation}
E=\{e_{11},\ldots,e_{1m_1},e_{21}, \ldots
,e_{1m_2-1},\ldots,e_{r1},\ldots ,e_{rm_r-1},e_1,\ldots,e_t\}
\end{equation}
where, $t=n-\sum\limits_{i=1}^{r} m_i+(r-1)$ and $\{e_{i1}, \ldots,
e_{iv}\}$ is the edge-set of $i$th-cycle such that $v=m_1$ for
$i=1,\;\; v=m_i-1$ for $i> 1$ and $e_{i1}$ always represents the
common edge between $i$th and $(i+1)$th-cycle (for $1\leq i< r$). We
give the characterization of $s(\mathcal{G}_{n,r}^{1})$ in
\ref{scn}. The formulation for $f-vectors$ is presented in \ref{fsc}
which further applied to device a formula to compute the {\em
Hilbert series} of the {\em face ring}
$k\big[\Delta_s(\mathcal{G}_{n,r}^{1})\big]$ (see \ref{Hil}).
Moreover in \ref{Ass}, we characterize of all the associated primes
of the facet ideal
$I_{\mathcal{F}}(\Delta_s(\mathcal{G}_{n,r}^{1}))$. Finally, we
prove that the face ring $k[\Delta_s(\mathcal{G}_{n,r}^{1})]$ is
Cohen-Macaulay in \ref{CM}.

\section{Background and basic notions}
In this section, we give some background and preliminaries of the
topic and define some notions that will be useful in the sequel.

\begin{Definition}\label{spa}
\em {A {\em spanning tree} of a simple connected finite graph
$G(V,E)$ is a subtree of $G$ that contains every vertex of $G$. We
represent the collection of all edge-sets of the spanning trees of
$G$ by $s(G)$, in other words;
$$s(G):=\{E(T_i)\subset E ,    \hbox {\, where $T_i$ is a spanning tree of $G$}\}.$$
}\end{Definition} \noindent For any simple connected graph $G$, the
authors mentioned the {\em cutting-down method} to obtain all the
spanning trees of $G$ in \cite{ARK}. According to this method a
spanning tree is obtained by removing one edge from each cycle
appearing in the graph. However, for the graph
$\mathcal{G}_{n,r}^{1}$ with $r$ cycles having one edge common in
every consecutive cycles and the labeling given in $(1)$, one can
obtain its spanning trees by removing exactly $r$ edges from the
graph with not more than two edges deleted from any cycle. Also,
keeping in view that if a common edge between two cycles is removed
then only one edge can be removed from the non common edges explicitly from the cycles on the either side of the common edge.\\
For example by using the above said {\em cutting-down method} for
the graph $\mathcal{G}_{10,2}^{1}$ given in fig. $1$:\\
$s(\mathcal{G}_{10,2}^{1})=\big\{ \{ e_1, e_2, e_3, e_4, e_{13},
e_{11}, e_{23}, e_{21}\}, \{ e_1, e_2, e_3, e_4, e_{13}, e_{11},
e_{23}, e_{22}\},\\ \{ e_1, e_2, e_3, e_4, e_{13}, e_{11}, e_{21},
e_{22}\}, \{ e_1, e_2, e_3, e_4, e_{13}, e_{23}, e_{21}, e_{22}\},
\{ e_1, e_2, e_3, e_4, e_{12}, e_{11}, e_{23}, e_{21}\},\\ \{ e_1,
e_2, e_3, e_4, e_{12}, e_{11}, e_{23}, e_{22}\}, \{ e_1, e_2, e_3,
e_4, e_{12}, e_{11}, e_{21}, e_{22}\}, \{ e_1, e_2, e_3, e_4,
e_{12}, e_{23}, e_{21}, e_{22}\},\\ \{ e_1, e_2, e_3, e_4, e_{13},
e_{12}, e_{23}, e_{22}\},\{ e_1, e_2, e_3, e_4, e_{13}, e_{12},
e_{23}, e_{21}\},\{ e_1, e_2, e_3, e_4, e_{13}, e_{12}, e_{21},
e_{22}\},\\ \{ e_1, e_2, e_3, e_4, e_{13}, e_{23},e_{21}, e_{22}\},
\{ e_1, e_2, e_3, e_4, e_{12}, e_{23},e_{21}, e_{22}\}\big\}$

\begin{center}
\begin{picture}(300,80)\label{fig}
\thicklines
\put(100,27){\line(1,0){69.6}}\put(100,60){\line(1,0){70}}
\put(96,25){$\bullet$}
\put(77,27){\line(1,0){30}}\put(75,24){$\bullet$}\put(85,18){${e_{13}}$}\put(104,40){${e_{11}}$}

\put(57,27){\line(1,0){30}}\put(77,27){\line(2,3){23}}\put(55,24){$\bullet$}\put(65,18){${e_{1}}$}
\put(99,27){\line(0,1){33}}\put(97,57){$\bullet$}
\put(70,40){$e_{12}$}\put(172,40){$e_{21}$}\put(200,40){$e_{3}$}
\put(132,18){${e_{23}}$}\put(182,18){${e_{2}}$}\put(205,18){${e_{4}}$}
\put(132,64){$e_{22}$} \put(165.5,25){$\bullet$}
\put(166,27){\line(1,0){33}}\put(195,24.5){$\bullet$}
\put(197,27){\line(1,0){33}}\put(225,24.5){$\bullet$}
\put(197,27){\line(0,1){30}}\put(195,54){$\bullet$}

\put(168,27){\line(0,1){33}}\put(165.5,57){$\bullet$}
\put(116,-4){\tiny Fig. 1 . $\mathcal{G}_{10,2}^1$}
\end{picture}
\end{center}

\begin{Definition}{\em
A {\em simplicial complex} $\Delta$ over a finite set $[n]=\{1,
2,\ldots,n \}$ is a collection of subsets of $[n]$, with the
property that $\{i\}\in \Delta$ for all $i\in[n]$, and if $F\in
\Delta$  then $\Delta$ will contain all the subsets of $F$
(including the empty set). An element of $\Delta$ is called a face
of $\Delta$, and the dimension of a face $F$ of $\Delta$ is defined
as $|F|-1$, where $|F|$ is the number of vertices of $F$. The
maximal faces of $\Delta$ under inclusion are called facets of
$\Delta$. The dimension of the simplicial complex $\Delta$ is :
$$\hbox{dim} \Delta = \max\{\hbox{dim} F | F \in \Delta\}.$$
We denote the simplicial complex $\Delta$ with facets $\{F_1,\ldots
, F_q\}$ by $$\Delta = \big\langle F_1,\ldots, F_q\big\rangle $$ }
\end{Definition}
\begin{Definition}\label{fvec}
{\em For a simplicial complex $\Delta$ over $[n]$ having dimension
$d$, its $f-vector$ is a $d+1$-tuple, defined as:
$$f(\Delta)=(f_0,f_1,\ldots,f_d)$$
where $f_i$ denotes the number of $i-dimensional$ faces in $\Delta.$
}\end{Definition}

\begin{Definition}\label{co}{\bf (Spanning Simplicial Complex )}\\
{\em Let $G(V,E)$ be a a simple finite connected graph and
$s(G)=\{E_1, E_2,\ldots,E_t\}$ be the edge-set of all possible
spanning trees of $G(V,E)$, then we defined (in \cite{ARK}) a
simplicial complex $\Delta_s(G)$ on $E$ such that the facets of
$\Delta_s(G)$ are precisely the elements of $s(G)$, we call
$\Delta_s(G)$ as the {\em spanning simplicial complex} of $G(V,E)$.
In other words;
$$\Delta_s(G)=\big\langle E_1,E_2,\ldots,E_t\big\rangle.$$
}\end{Definition}
 Here we recall a definition from \cite{F1}.
\begin{Definition}{\em
Let $\Delta$ be a simplicial complex with vertex set $V=[n]$ and
facets $F_1,F_2,\ldots,F_q$. A {\em vertex cover} for $\Delta$ is a
subset $A$ of $V$ such that $A\cap F_i\neq\emptyset$ for all
$i\in\{1,2,\ldots,q\}$. A {\em minimal vertex cover} of $\Delta$ is
a subset $A$ of $V$ such that $A$ is a {\em vertex cover}, and no
proper subset of $A$ is a {\em vertex cover} for $\Delta$.}
\end{Definition}
For example, the {\em minimal vertex covers} for the {\em spanning
simplicial complex} $\Delta_s(\mathcal{G}_{10,2}^1)$ given in Fig.
1, are as follows:
$$\{e_1\},\{e_2\},\{e_3\},\{e_4\},\{e_{13},e_{12}\},\{e_{23},e_{22}\}$$

\section{Spanning trees of $\mathcal{G}_{n,r}^1$ and Face ring $\Delta_s(\mathcal{G}_{n,r}^1)$ }

In this section, we discuss the combinatorial properties of
$\mathcal{G}_{n,r}^1$.  We use $\bf \tau(\mathcal{G}_{n,r}^1)$ to
denote the {\bf total number of cycles} contained in
$\mathcal{G}_{n,r}^1$. We begin with the elementary result, that
tells the total number of cycles contained by $\mathcal{G}_{n,r}^1$.
\begin{Proposition}\label{cycles}{\em
The total number of cycles in the graph $\mathcal{G}_{n,r}^1$ will
be
$$\tau(\mathcal{G}_{n,r}^1)=\frac{r(r+1)}{2}$$
}\end{Proposition}
\begin{proof}
As the graph $\mathcal{G}_{n,r}^{1}$ contains one-edge connected
chain $\mathbb{C}_r^1$ of $r$ cycles $\{C_1, C_2,\ldots, C_r\}$. By
removing the common edges between any number of consecutive cycles,
we obtain a cycle by the remaining edges. The cycle obtained in this
way by adjoining consecutive cycles $C_i,C_{i+1},\ldots,C_{i+k}$ is
denoted by $\bf C_{i,i+1,\ldots,i+k}$. Therefore, we get the
following cycles
$$C_{1,2},C_{2,3},\ldots,C_{r-1,r},C_{1,2,3},\ldots,C_{r-2,r-1,r},\ldots,C_{2,3,\ldots,r},C_{12,3,\ldots,r}$$
Hence, the set of all possible cycles contained in the graph
$\mathcal{G}_{n,r}^{1}$ will be
$$\{C_{i,i+1,\ldots,i+k}\mid \;\;\;i\in\{1,2,\ldots,r-k\}\;\hbox{and}\; 0 \le k\le r-1\}.$$
Therefore, we get the total number of cycles contained in the graph
$\mathcal{G}_{n,r}^{1}$ as
$$\tau(\mathcal{G}_{n,r}^{1})=\sum\limits_{k=0}^{r-1}{\sum\limits_{i=1}^{r-k}1}=\frac{r(r+1)}{2}.$$\end{proof}
It is clear from above proposition that the cycle
$C_{i,i+1,\ldots,i+k}$ is obtained by removing the common edges in
between the adjacent cycles $C_i, C_{i+1},\ldots, C_{i+k}$. We
denote the {\bf length of cycle $\bf C_{i,i+1,\ldots, i+k}$} by $\bf
|C_{i,i+1,\ldots,i+k}|$.
\begin{Proposition}\label{scn}
\em{Let $\mathcal{G}_{n,r}^{1}$ be a graph containing the one-edge
connected chain $\mathbb{C}_r^1$ of $r$ cycles $\{C_1, C_2,\ldots,
C_r\}$, then the length of cycle $C_{i,i+1,\ldots,i+k}$ will be
$$\Big|C_{i,i+1,\ldots,i+k}\Big|=\sum\limits_{\alpha=0}^{k}\big|C_{i+\alpha}\big|-2k.$$
}\end{Proposition}
\begin{proof}
It is clear from above that $C_{i,i+1,\ldots,i+k}$ is obtained by
deleting the common edges shared by the adjacent cycles $\{C_i,
C_{i+1},\ldots, C_{i+k}\}$ in $\mathcal{G}_{n,r}^{1}$. Therefore,
the length of the cycle $C_{i,i+1,\ldots,i+k}$ is obtained by adding
lengths of all $C_i,C_{i+1},\ldots,C_{i+k}$ and subtracting $2k$
from it, since the common edges are being counted twice. Hence, we
have
$$
\Big|C_{i,i+1,\ldots,i+k}\Big|=\sum\limits_{\alpha=0}^{k}\big|C_{i+\alpha}\big|-2k.
$$
\end{proof}
We use $\bf |C_{i,i+1,\ldots,i+k}\bigcap C_{j,j+1,\ldots,j+l}|$ to
denote the {\bf number of edges shared by the cycles}
$C_{i,i+1,\ldots,i+k}$ and $C_{j,j+1,\ldots,j+l}$. The following
proposition characterizes $|C_{i,i+1,\ldots,i+k}\bigcap
C_{j,j+1,\ldots,j+l}|$ in $\mathcal{G}_{n,r}^{1}$.
\begin{Proposition}\label{scn}
\em{Let $\mathcal{G}_{n,r}^{1}$ be a graph containing the one-edge
connected chain $\mathbb{C}_r^1$ of $r$ cycles $\{C_1, C_2,\ldots,
C_r\}$ of lengths $m_1, m_2,\ldots, m_r$, then for $1\le k\le l\le
r$ we have
$$\Big|C_{i,i+1,\ldots,i+k}\bigcap C_{j,j+1,\ldots,j+l}\Big|=\left\{
                                                       \begin{array}{ll}
                                                         1, & {i+k=j-1;} \\
                                                         |C_{j,j+1,\ldots,j+\alpha}|-2, & {i+k=j+\alpha ;\; 0\le\alpha \le k-1;} \\
                                                         |C_{j,j+1,\ldots,j+k}|-1, & {i+k=j+k;} \\
                                                         |C_{j,j+1,\ldots,j+l}|, & {i+k=j+l\;and\; l=k;} \\
                                                         |C_{i,i+1,\ldots,i+k}|-2, & {i+k = j+\alpha; \; k+1 \le \alpha \le l-1;} \\
                                                         |C_{i,i+1,\ldots,i+k}|-1, & {i+k = j+l;} \\
                                                         |C_{i,i+1,\ldots,i+k-\alpha}|-2, & {i+k=j+l+\alpha; \; 1\le\alpha\le k;} \\
                                                         1, & {i=j+l+1;} \\
                                                         0, & {otherwise.}
                                                       \end{array}
                                                     \right.$$}
\end{Proposition}
\begin{proof}
Here we denote $m_{i,i+1,\ldots,i+k}=\Big|C_{i,i+1,\ldots,i+k}\Big|$.\\
Now for $1\le k\le l\le
r$ we discuss the following cases for $\Big|C_{i,i+1,\ldots,i+k}\bigcap C_{j,j+1,\ldots,j+l}\Big|$ :\\
\begin{description}
          \item[Case (i)] If $i+k=j-1$, then the right most edges of the cycle $C_{i,i+1,\ldots,i+k}$ are from its adjoining cycle $C_{i+k}$ and the left most edges of the cycle $C_{j,j+1,\ldots,j+l}$ are from its adjoining cycle $C_{j+l}$, and since $C_{i+k}$ and $C_{j+l}$ are consecutive so they have only one edge in common.
          \item[Case (ii)] If $i+k=j+\alpha ;\; 0\le\alpha \le k-1$, then the left most $\alpha$ adjoining cycles of the cycles of $C_{j,j+1,\ldots,j+l}$, i.e., $C_j,C_{j+1},\ldots,C_{j+\alpha}$ coincide with the right most $\alpha$ adjoining cycles of the cycles of $C_{i,i+1,\ldots,i+k}$. Therefore, the intersection $C_{i,i+1,\ldots,i+k}\bigcap C_{j,j+1,\ldots,j+l}$ will contain all edges of $C_{j,j+1,\ldots,j+\alpha}$ except its two edges, one the edge of $C_{j,j+1,\ldots,j+\alpha}$ which is the common edge of $C_{j+\alpha}$ and $C_{j+\alpha+1}$ and second the edge of $C_{j,j+1,\ldots,j+\alpha}$ which is common edge between $C_j$ and $C_{j-1}$.
          \item[Case (iii)] If $i+k=j+k\; and\; k<l$, then $i=j$ and therefore the cycle $C_{i,i+1,\ldots,i+k}$ lies completely in the cycle $C_{j,j+1,\ldots,j+l}$ except its one edge which is the common edge between its adjoining cycle $C_{j+k}$ and the cycle $C_{j+k+1}$. Therefore the intersection $C_{i,i+1,\ldots,i+k}\bigcap C_{j,j+1,\ldots,j+l}$ will contain all edges of $C_{j,j+1,\ldots,j+k}$ except one.
          \item[Case (iv)] If $i+k=j+l+\alpha; \; 1\le\alpha\le k$, then the left most $k-\alpha+1$ adjoining cycles of the cycle $C_{i,i+1,\ldots,i+k}$, which are $C_i,C_{i+1},\ldots,C_{i+k-\alpha}$ coincide with the right most $k-\alpha+1$ adjoining cycles of the cycle $C_{j,j+1,\ldots,j+l}$. Hence the intersection $C_{i,i+1,\ldots,i+k}\bigcap C_{j,j+1,\ldots,j+l}$ will contain all the edges of the cycle $C_{i,i+1,\ldots,i+k-\alpha}$ except two; one is the common edge of its adjoining cycle $C_i$ and the cycle $C_{i-1}$ and the other is the common edge of its adjoining cycle $C_{i+k-\alpha}$ and the cycle $C_{i+k-\alpha+1}$.
          \end{description}
The remaining cases can be proved in the similar way.
\end{proof}

\begin{Lemma}\label{scn}{\bf Characterization of $s(\mathcal{G}_{n,r}^1)$}\\
\em{Let $\mathcal{G}_{n,r}^1$ be the $r-$cycles graph with edge set $E$ as defined in eq (1), then a subset $E(T_{(j_1i_1,j_2i_2,\hdots,j_ri_r)})\subset E$, where $j_\alpha\in\{1,2,\hdots,r\}$, $i_\alpha\in\{1,2,\hdots,m_{j_\alpha}-1\};\;j_\alpha\geq 2$ and $i_\alpha\in\{1,2,\hdots,m_1\};\;j_\alpha=1$, will belong to $s(\mathcal{G}_{n,r}^1)$ if and only if it satisfies any of the following:
\begin{enumerate}
  \item if $j_\alpha i_\alpha\neq j_\alpha 1$ for all $\alpha$ except for which $j_\alpha i_\alpha=ri_\alpha$, then $E(T_{(j_1i_1,j_2i_2,\hdots,j_ri_r)})=E\setminus\{e_{1i_1},e_{2i_2},\hdots, e_{ri_r}\}$
  \item if $j_\alpha i_\alpha= j_\alpha 1$ for any $\alpha$, then $E(T_{(j_1i_1,j_2i_2,\hdots,j_ri_r)})=E\setminus\{e_{j_1i_1},e_{j_2i_2},\hdots, e_{j_ri_r}\}$ where, $\{e_{j_1i_1},e_{j_2i_2},\hdots,e_{j_{\alpha -1}i_{\alpha -1}},e_{j_{\alpha +1}i_{\alpha +1}},\hdots , e_{j_ri_r}\}$ will contain exactly one edge from $C_{j_{\alpha}(j_{\alpha}+1)}\setminus \{ e_{(j_{\alpha}-1)1}, e_{j_{\alpha}1} \}$.
  \item if $j_{\alpha} i_{\alpha}= j_{\alpha} 1$ for  $\alpha \in \{r_1,r_1+1,\hdots,r_2\}$, where $1\le r_1<r_2< r$ then
        \begin{enumerate}
            \item if $e_{j_{r_1}1},e_{j_{(r_1+1)}1},\hdots, e_{j_{r_2}1}$ are common edges from consecutive cycles then $E(T_{(j_1i_1,j_2i_2,\hdots,j_ri_r)})=E\setminus\{e_{j_1i_1},e_{j_2i_2},\hdots, e_{j_ri_r}\}$ such that $\{e_{j_1i_1},e_{j_2i_2},\hdots, e_{j_ri_r}\}\setminus \{e_{j_{r_1}1},e_{j_{(r_1+1)}1},\hdots, e_{j_{r_2}1}\}$ will contain exactly one edge from $C_{j_{r_1}j_{(r_1+1)},\hdots,j_{r_2}}\setminus \{e_{(j_{r_1}-1)1},e_{j_{r_2}1} \}$,
            \item if none of $e_{j_{r_1}1},e_{j_{(r_1+1)}1},\hdots, e_{j_{r_2}1}$ are common edges from consecutive cycles then $E(T_{(j_1i_1,j_2i_2,\hdots,j_ri_r)})=E\setminus\{e_{j_1i_1},e_{j_2i_2},\hdots, e_{j_ri_r}\}$ such that for each edge $e_{j_{r_t}1}$ case 2 holds.
            \item if some of $e_{j_{r_1}1},e_{j_{(r_1+1)}1},\hdots, e_{j_{r_2}1}$ are common edges from consecutive cycles then $E(T_{(j_1i_1,j_2i_2,\hdots,j_ri_r)})=E\setminus\{e_{j_1i_1},e_{j_2i_2},\hdots, e_{j_ri_r}\}$ such that $(3.(a))$ is satisfied for the common edges of consecutive cycles and $(3.(b))$ is satisfied for remaining common edges.
        \end{enumerate}
\end{enumerate}

In particular, if we denote the above classes of subsets of $E$ by $\mathcal{C}_{(1)},\mathcal{C}_{(2)},\mathcal{C}_{(3a)},\mathcal{C}_{(3b)},\mathcal{C}_{(3c)}$ respectively then,
$$ s(\mathcal{G}_{n,r}^1)=\mathcal{C}_{(1)}\bigcup\mathcal{C}_{(2)}\bigcup\mathcal{C}_{(3a)}\bigcup\mathcal{C}_{(3b)}\bigcup\mathcal{C}_{(3c)}$$}
\end{Lemma}

\begin{proof}
Since $\mathcal{G}_{n,r}^1$ is a $r$-cycles graph with cycles $C_1,C_2,\hdots,C_r$ and $e_{11},e_{21},\hdots,e_{(r-1)1}$ as common edges between consecutive cycles and by cutting down process a total of $r$ edges must be removed with not more than one edges from the non common edges of each cycle. Therefore, in order to obtain a spanning tree of $\mathcal{G}_{n,r}^1$ with none of common edges $e_{11},e_{21},\hdots,e_{(r-1)1}$ to be removed, we need to remove exactly one edge from the non common edges from each cycle. This explains the case (1) of the above lemma.

Now for a spanning tree of $\mathcal{G}_{n,r}^1$ such that exactly one common edge $e_{j_\alpha 1}$ is removed, we need to remove precisely $r-1$ edges using cutting down process from the remaining edges. However, from the non common edges of the cycle $C_{j_{\alpha}(j_{\alpha}+1)}$ , we cannot remove more than one edge (since that will result in a disconnected graph). This explains the proof of case of (2) of the lemma.

Next for the case (3.a), we need to obtain a spanning tree of $\mathcal{G}_{n,r}^1$  such that $r_2-r_1$ common edges must be removed from consecutive cycles. If $C_{j_{r_1}},C_{j_{r_1+1}},\hdots,C_{j_{r_2}}$ are consecutive cycles then the remaining $r-(r_1-r_2)$ edges must be removed in such a way that exactly one edge is removed from the non common edges of $C_{j_{r_1}j_{(r_1+1)},\hdots,j_{r_2}}$  and the remaining $r-(r_1-r_2)$ cycles of the graph $\mathcal{G}_{n,r}^1$ which concludes the case.

The remaining cases of the lemma can be visualised in similar manner using the above cases . Consequently, if we denote the above disjoint classes of subsets of $E$ by $\mathcal{C}_{(1)},\mathcal{C}_{(2)},\mathcal{C}_{(3a)},\mathcal{C}_{(3b)},\mathcal{C}_{(3c)}$ respectively, then, we get the desired result for $s(\mathcal{G}_{n,r}^1)$ as follows:
$$ s(\mathcal{G}_{n,r}^1)=\mathcal{C}_{(1)}\bigcup\mathcal{C}_{(2)}\bigcup\mathcal{C}_{(3a)}\bigcup\mathcal{C}_{(3b)}\bigcup\mathcal{C}_{(3c)}$$

\end{proof}

Our next result is the characterization of the $f$-vector of $\Delta_s(\mathcal{G}_{n,r}^1)$.

\begin{Proposition}\label{fsc}
  \em{Let $\Delta_s(\mathcal{G}_{n,r}^1)$ be a spanning simplicial complex of the graph $\mathcal{G}_{n,r}^1$, then the $dim(\Delta_s(\mathcal{G}_{n,r}^1))=n-r-1$ with $f-$vector $f(\Delta_s(\mathcal{G}_{n,r}^1))=(f_0,f_1,\hdots,f_{n-r-1})$ and\\
$  f_i=\left(
       \begin{array}{c}
         n \\
         i+1 \\
       \end{array}
     \right)+\sum\limits_{k=1}^{\tau}(-1)^k\\
\left[
                             \begin{array}{c}
                               {\sum\limits_{j=1}^{k}\sum\limits_{k_{s_j}=0}^{r-1}\sum\limits_{i_{s_j}=1
}^{r-k_{s_j}} \left(
                                       \begin{array}{c}
                                         n-\sum\limits_{j=1}^{k} m_{i_{s_j},i_{s_j}+1,\hdots,i_{s_j}+k_{s_j}}+\sum\limits_{u,v=1}^{k}\big|C_{i_{s_u},i_{s_u}+1,\hdots,i_{s_u}+k_{s_u}}\bigcap C_{i_{s_v},i_{s_v}+1,\hdots,i_{s_v}+k_{s_v}}\big| \\
                                         i+1-\sum\limits_{j=1}^{k} m_{i_{s_j},i_{s_j}+1,\hdots,i_{s_j}+k_{s_j}}+\sum\limits_{u,v=1}^{k}\big|C_{i_{s_u},i_{s_u}+1,\hdots,i_{s_u}+k_{s_u}}\bigcap C_{i_{s_v},i_{s_v}+1,\hdots,i_{s_v}+k_{s_v}}\big| \\
                                       \end{array}
                                     \right)}\\
                             \end{array}
                           \right]$
\\where $0\le i\le n-r-1$}\end{Proposition}

\begin{proof}
  Let $E$ be the edge set of $\mathcal{G}_{n,r}^1$ and $\mathcal{C}_{(1)},\mathcal{C}_{(2)},\mathcal{C}_{(3a)}, \mathcal{C}_{(3b)}, \mathcal{C}_{(3c)}$ are disjoint classes of spanning trees of $\mathcal{G}_{n,r}^1$ then from lemma 3.2 we have

$$ s(\mathcal{G}_{n,r}^1)=\mathcal{C}_{(1)}\bigcup\mathcal{C}_{(2)}\bigcup\mathcal{C}_{(3a)}\bigcup\mathcal{C}_{(3b)}\bigcup\mathcal{C}_{(3c)}$$

Therefore, by definition 2.4 we can write
$$ \Delta_s (\mathcal{G}_{n,r}^1)=\Big\langle\mathcal{C}_{(1)}\bigcup\mathcal{C}_{(2)}\bigcup\mathcal{C}_{(3a)}\bigcup\mathcal{C}_{(3b)}\bigcup\mathcal{C}_{(3c)} \Big\rangle$$

Since each facet $\hat{E}_{(j_1i_1,j_2i_2,\hdots,j_ri_r)}=E(T_{(j_1i_1,j_2i_2,\hdots,j_ri_r)})$ is obtained by deleting exactly $r$ edges from the edge set of $\mathcal{G}_{n,r}^1$, keeping in view lemma 3.2, therefore dimension of each facet is same i.e., $n-r-1$ ( since $ |\hat{E}_{(j_1i_1,j_2i_2,\hdots,j_ri_r)}|=n-r$ ) and hence dimension of $\Delta_s(\mathcal{G}_{n,r}^1)$ will be $n-r-1$.\\
Also it is clear from the definition of $\Delta_s(\mathcal{G}_{n,r}^1)$ that it contains all those subsets of $E$ which do not contain the sets $\{e_{11},\hdots,e_{1m_1}\}$ and $\{e_{(i-1)1},e_{i1},\hdots,e_{im_i-1}\}$ for all $2\le i\le r$, i.e., those subsets of $E$ which do not contain any cycle in the graph $\mathcal{G}_{n,r}^1$.\\

Now by lemma 3.1 the total cycles in the graph $\mathcal{G}_{n,r}^1$ are $C_{i,i+1,\hdots,i+k}\;\;\;i\in\{1,2,\hdots,r-k\}\;and\; 0 \le k\le r-1$, and their total number is $\tau$. Let $F$ be any subset of $E$ of order $i+1$ such that it does not contain any $C_{i,i+1,\hdots,i+k}\;\;\;i\in\{1,2,\hdots,r-k\}\;and\; 0 \le k\le r-1$, in it. The total number of such $F$ is indeed $f_i$. We use inclusion exclusion principle to find this number. Therefore,\\
$f_i=$ Total number of subsets of $E$ of order $i+1$ not containing $C_{i,i+1,\hdots,i+k}; \;i\in\{1,2,\hdots,r-k\}\;and\; 0 \le k\le r-1$.\\

By Inclusion Exclusion Principle we have,\\

$f_i=\Big( $ Total number of subsets of $E$ of order $i+1\Big)- \Big(\sum\limits_{j=1}^{1}\sum\limits_{k_{s_j}=0}^{r-1}\sum\limits_{i_{s_j}=1}^{r-k{s_j}}$ number of subsets of $E$  of order $i+1$ containing $C_{i_{s_j},i_{s_j}+1,\hdots,i_{s_j}+k_{s_j}}\Big)+\Big(\sum\limits_{j=1}^{2}\sum\limits_{k_{s_j}=0}^{r-1}\sum\limits_{i_{s_j}=1}^{r-k_{s_j}}$ number of subsets of $E$ of order $i+1$ containing both $C_{i_{s_j},i_{s_j}+1,\hdots,i_{s_j}+k_{s_j}}\Big)- \cdots +(-1)^{\tau}\Big(\sum\limits_{j=1}^{\tau}\sum\limits_{k_{s_j}=0}^{r-1}\sum\limits_{i_{s_j}=1
}^{r-k_{s_j}}$ number of subsets of $E$ of order $i+1$ containing each $C_{i_{s_j},i_{s_j}+1,\hdots,i_{s_j}+k_{s_j}}\Big)$
\\
\\This implies\\
$f_i=\left(
       \begin{array}{c}
         n \\
         i+1 \\
       \end{array}
     \right)-\Big[\sum\limits_{j=1}^{1}\sum\limits_{k_{s_j}=0}^{r-1}\sum\limits_{i_{s_j}=1}^{r-k_{s_j}}\left(
                                       \begin{array}{c}
                                         n-m_{i_{s_1},i_{s_1}+1,\hdots,i_{s_1}+k_{s_1}} \\
                                         i+1-m_{i_{s_1},i_{s_1}+1,\hdots,i_{s_1}+k_{s_1}} \\
                                       \end{array}
                                     \right)
     \Big]+\\ \left[
                \begin{array}{c}
                  \sum\limits_{j=1}^{2}\sum\limits_{k_{s_j}=0}^{r-1}\sum\limits_{i_{s_j}=1}^{r-k_{s_j}}
                                       \left(
                                       \begin{array}{c}
                                         n-\sum\limits_{j=1}^{2} m_{i_{s_j},i_{s_j}+1,\hdots,i_{s_j}+k_{s_j}}+\sum\limits_{u,v=1}^{2}\big|C_{i_{s_u},i_{s_u}+1,\hdots,i_{s_u}+k_{s_u}}\bigcap C_{i_{s_v},i_{s_v}+1,\hdots,i_{s_v}+k_{s_v}}\big| \\
                                         i+1-\sum\limits_{j=1}^{2} m_{i_{s_j},i_{s_j}+1,\hdots,i_{s_j}+k_{s_j}}+\sum\limits_{u,v=1}^{2}\big|C_{i_{s_u},i_{s_u}+1,\hdots,i_{s_u}+k_{s_u}}\bigcap C_{i_{s_v},i_{s_v}+1,\hdots,i_{s_v}+k_{s_v}}\big| \\
                                       \end{array}
                                     \right) \\
                \end{array}
              \right]\\
 - \cdots+(-1)^{\tau}\\
\left[
                             \begin{array}{c}
                               {\sum\limits_{j=1}^{\tau}\sum\limits_{k_{s_j}=0}^{r-1}\sum\limits_{i_{s_j}=1
}^{r-k_{s_j}} \left(
                                       \begin{array}{c}
                                         n-\sum\limits_{j=1}^{\tau} m_{i_{s_j},i_{s_j}+1,\hdots,i_{s_j}+k_{s_j}}+\sum\limits_{u,v=1}^{\tau}\big|C_{i_{s_u},i_{s_u}+1,\hdots,i_{s_u}+k_{s_u}}\bigcap C_{i_{s_v},i_{s_v}+1,\hdots,i_{s_v}+k_{s_v}}\big| \\
                                         i+1-\sum\limits_{j=1}^{\tau} m_{i_{s_j},i_{s_j}+1,\hdots,i_{s_j}+k_{s_j}}+\sum\limits_{u,v=1}^{\tau}\big|C_{i_{s_u},i_{s_u}+1,\hdots,i_{s_u}+k_{s_u}}\bigcap C_{i_{s_v},i_{s_v}+1,\hdots,i_{s_v}+k_{s_v}}\big| \\
                                       \end{array}
                                     \right)}\\
                             \end{array}
                           \right]$
\\This implies\\
$  f_i=\left(
       \begin{array}{c}
         n \\
         i+1 \\
       \end{array}
     \right)+\sum\limits_{k=1}^{\tau}(-1)^k\\
\left[
                             \begin{array}{c}
                               {\sum\limits_{j=1}^{k}\sum\limits_{k_{s_j}=0}^{r-1}\sum\limits_{i_{s_j}=1
}^{r-k_{s_j}} \left(
                                       \begin{array}{c}
                                         n-\sum\limits_{j=1}^{k} m_{i_{s_j},i_{s_j}+1,\hdots,i_{s_j}+k_{s_j}}+\sum\limits_{u,v=1}^{k}\big|C_{i_{s_u},i_{s_u}+1,\hdots,i_{s_u}+k_{s_u}}\bigcap C_{i_{s_v},i_{s_v}+1,\hdots,i_{s_v}+k_{s_v}}\big| \\
                                         i+1-\sum\limits_{j=1}^{k} m_{i_{s_j},i_{s_j}+1,\hdots,i_{s_j}+k_{s_j}}+\sum\limits_{u,v=1}^{k}\big|C_{i_{s_u},i_{s_u}+1,\hdots,i_{s_u}+k_{s_u}}\bigcap C_{i_{s_v},i_{s_v}+1,\hdots,i_{s_v}+k_{s_v}}\big| \\
                                       \end{array}
                                     \right)}\\
                             \end{array}
                           \right]
$

\end{proof}

\begin{cor}
Let $\Delta_s(\mathcal{G}_{n,2}^1)$ be a spanning simplicial complex of a graph with $2$ cycles of lengths $m_1, m_2$ having one edge common, then the $dim(\Delta_s(\mathcal{G}_{n,2}^1))=n-3$ with $f-$vectors $f(\Delta_s(G_{n,2}^1))=(f_0,f_1,\ldots,f_{n-3})$ and
\\$  f_i=\left(
       \begin{array}{c}
         n \\
         i+1 \\
       \end{array}
     \right)-\Big[\left(
                    \begin{array}{c}
                      n-m_1 \\
                      i+1-m_1 \\
                    \end{array}
                  \right)+\left(
                    \begin{array}{c}
                      n-m_2 \\
                      i+1-m_2 \\
                    \end{array}
                  \right)+\left(
                    \begin{array}{c}
                      n-m_{1,2} \\
                      i+1-m_{1,2} \\
                    \end{array}
                  \right)\Big]+\\ \Big[\left(
                                      \begin{array}{c}
                                        n-m_1-m_2+|C_1\cap C_2| \\
                                        i+1-m_1-m_2+|C_1\cap C_2| \\
                                      \end{array}
                                    \right)+\left(
                                      \begin{array}{c}
                                        n-m_1-m_{1,2}+|C_1\cap C_{1,2}| \\
                                        i+1-m_1-m_{1,2}+|C_1\cap C_{1,2}| \\
                                      \end{array}
                                    \right)+\\ \left(
                                      \begin{array}{c}
                                        n-m_2-m_{1,2}+|C_2\cap C_{1,2}| \\
                                        i+1-m_2-m_{1,2}+|C_2\cap C_{1,2}| \\
                                      \end{array}
                                    \right)
\Big]+\\ \Big[\left(
                                      \begin{array}{c}
                                        n-m_1-m_2-m_{1,2}+|C_1\cap C_2|+|C_1\cap C_{1,2}|+|C_2\cap C_{1,2}| \\
                                        i+1-m_1-m_2-m_{1,2}+|C_1\cap C_2|+|C_1\cap C_{1,2}|+|C_2\cap C_{1,2}| \\
                                      \end{array}
                                    \right)
\Big]$
\\where $0\le i\le n-3.$
\end{cor}

For a simplicial complex $\Delta$ over $[n]$, one would associate to it the
Stanley-Reisner ideal, that is, the monomial ideal $I_{\mathcal N}(\Delta)$ in
$S=k[x_1, x_2,\ldots ,x_n]$ generated by monomials corresponding to
non-faces of this complex (here we are assigning one variable of the
polynomial ring to each vertex of the complex). It is well known
that the face ring $k[\Delta]=S/I_{\mathcal N}(\Delta)$
is a standard graded algebra. We refer the readers to \cite{HP} and
\cite{Vi} for more details about graded algebra $A$, the Hilbert
function $H(A,t)$ and the Hilbert series $H_t(A)$ of a graded algebra.

Our main result of this section is as follows;
\begin{Theorem}\label{Hil} {\em Let $\Delta_s(\mathcal{G}_{n,r}^1) $ be the spanning simplicial complex of
$\mathcal{G}_{n,r}^1$, then the Hilbert series of the face ring
$k\big[\Delta_s(\mathcal{G}_{n,r}^1)\big]$ is given by,\\
$H(k[\Delta_s(\mathcal{G}_{n,r}^1)],t)=1+\sum\limits_{i=0}^{d}\frac{{n\choose
{i+1}}{t^{i+1}}}{(1-t)^{i+1}}+\sum\limits_{i=0}^{d}\sum\limits_{k=1}^{\tau}(-1)^k\\
\tiny{
 \Big[
                             \begin{array}{c}
                               {\sum\limits_{j=1}^{k}\sum\limits_{k_{s_j}=0}^{r-1}\sum\limits_{i_{s_j}=1
}^{r-k_{s_j}} \left(
                                       \begin{array}{c}
                                         n-\sum\limits_{j=1}^{k} m_{i_{s_j},i_{s_j}+1,\hdots,i_{s_j}+k_{s_j}}+\sum\limits_{u,v=1}^{k}\big|C_{i_{s_u},i_{s_u}+1,\hdots,i_{s_u}+k_{s_u}}\bigcap C_{i_{s_v},i_{s_v}+1,\hdots,i_{s_v}+k_{s_v}}\big| \\
                                         i+1-\sum\limits_{j=1}^{k} m_{i_{s_j},i_{s_j}+1,\hdots,i_{s_j}+k_{s_j}}+\sum\limits_{u,v=1}^{k}\big|C_{i_{s_u},i_{s_u}+1,\hdots,i_{s_u}+k_{s_u}}\bigcap C_{i_{s_v},i_{s_v}+1,\hdots,i_{s_v}+k_{s_v}}\big| \\
                                       \end{array}
                                     \right)}\\
                             \end{array}
                           \Big]} \frac{t^{i+1}}{(1-t)^{i+1}}$}
\end{Theorem}

\begin{proof}
From \cite{Vi}, we know that if $\Delta$ is a simplicial complex of
dimension $d$ and $f(\Delta)=(f_0, f_1, \ldots,f_d)$ its $f$-vector,
then the Hilbert series of face ring $k[\Delta]$ is given
by $$H(k[\Delta],t)= 1+\sum_{i=0}^{d}\frac{f_i
t^{i+1}}{(1-t)^{i+1}}.$$ By substituting the values of $f_i$'s from
Proposition \ref{fsc} in this above expression, we get the desired result.
\end{proof}

\section{Associated primes of the facet ideal $I_{\mathcal{F}}(\Delta_s(\mathcal{G}_{n,r}^1))$  }

We present the characterization of all associated
primes of the facet ideal $I_{\mathcal{F}}(\Delta_s(\mathcal{G}_{n,r}^1))$ of
spanning simplicial complex $\Delta_s(\mathcal{G}_{n,r}^1)$ in this section.

Associated to a simplicial complex $\Delta$ over $[n]$, one defines {\em the facet
ideal} $I_{\mathcal{F}}(\Delta)\subset S$, which is
generated by square-free monomials $x_{i1} \ldots x_{is}$, where
$\{v_{i1} ,\ldots, v_{is}\}$ is a facet of $\Delta$.

\begin{Lemma}\label{Ass}{\em
If $\Delta_s(\mathcal{G}_{n,r}^1)$ be the spanning simplicial complex of the
$r$-cycles graph $\mathcal{G}_{n,r}^1$, then

$$I_{\mathcal{F}}(\Delta_s(\mathcal{G}_{n,r}^1))=\left(\bigcap_{e_t\not\in C_{i}}^{ 1 \leq i\leq r}(x_t)\right)\bigcap \left(\bigcap_{ 2\leq i_\alpha\neq i_\beta\leq m_{j_\alpha}-1}^{2\leq j_\alpha\leq r-1}(x_{j_\alpha i_\alpha},x_{j_\alpha i_\beta})\right)$$
$$\bigcap\left(\bigcap_{2\leq i_\alpha\neq i_\beta\leq m_1}\left(x_{1i_\alpha},x_{1i_\beta,} \right)\right)\bigcap\left(\bigcap_{1\leq i_\alpha\neq i_\beta\leq m_r-1}\left(x_{ri_\alpha},x_{ri_\beta,} \right)\right)$$ }
\end{Lemma}

\begin{proof} Consider the spanning simplicial complex
$\Delta_s(\mathcal{G}_{n,r}^1)$ and let
$I_{\mathcal{F}}(\Delta_s(\mathcal{G}_{n,r}^1))$ be the facet ideal of
$\Delta_s(\mathcal{G}_{n,r}^1)$. Since from \cite[Proposition 1.8]{F1}, we know
that a minimal prime ideal of the facet ideal
$I_{\mathcal{F}}(\Delta)$ has one to one correspondence with the
minimal vertex cover of the simplicial complex. Therefore, in order
to compute the primary decomposition of the facet ideal
$I_{\mathcal{F}}(\Delta_s(\mathcal{G}_{n,r}^1))$; it is sufficient to compute
all the minimal vertex cover of $\Delta_s(\mathcal{G}_{n,r}^1)$.\\

Indeed clear from the definition of $\Delta_s(\mathcal{G}_{n,r}^1)$ and by Lemma \ref{scn} that $\{e_t\}$ is a minimal vertex cover of $\Delta_s(\mathcal{G}_{n,r}^1)$ such that $e_t\not\in
C_{i} \;\forall\; i\in \{1,\ldots, r\}$ . Moreover, $\{e_{j_\alpha i_\alpha},e_{j_\alpha i_\beta}\}$ is also a minimal vertex cover of $\Delta_s(\mathcal{G}_{n,r}^1)$ with
$2\leq i_\alpha\neq i_\beta\leq  m_{j_\alpha}-1$ for $ j_\alpha\in\{2,\hdots, r-1\}$ , $2\leq i_\alpha\neq i_\beta\leq  m_1$ for $ j_\alpha=1$ and $1\leq i_\alpha\neq i_\beta\leq  m_r-1$ for $j_\alpha=r$. Indeed for any $\hat{E}_{(j_1i_1,j_2i_2,\hdots, j_ri_r)}\in s(\mathcal{G}_{n,r}^1)$ the intersection $\{e_{j_\alpha i_\alpha},e_{j_\alpha i_\beta}\}\cap \hat{E}_{(j_1i_1,j_2i_2,\hdots, j_ri_r)} $ is nonempty.
\end{proof}
\section{Cohen-Macaulayness of the face ring of $\Delta_s(\mathcal{G}_{n,r}^1)$}
In this section, we include some definitions and results from \cite{AR} and use them to show that the face ring of $\Delta_s(\mathcal{G}_{n,r}^1)$ is Cohen-Macaulay.
\begin{Definition}\label{qlq}\cite{AR}\\
\em{Let $I\subset S=k[x_1,x_2,\hdots,x_n]$ be a monomial ideal, we
say that $I$ will have the \textit{quasi-linear quotients}, if there
exists a minimal monomial system of generators $m_1,m_2,\hdots,m_r$
such that $mindeg(\hat{I}_{m_i})=1$ for all $1<i\le r$, where
$$\hat{I}_{m_i}=(m_1,m_2,\hdots,m_{i-1}):(m_i).$$}
\end{Definition}
\begin{Theorem}\cite{AR}
\em{Let $\Delta$ be a pure simplicial complex of dimension $d$ over
$[n]$. Then $\Delta$ will be a shellable simplicial complex if and
only if $I_{\mathcal{F}}(\Delta)$ will have the quasi-linear
quotients.}
\end{Theorem}
\begin{cor}\label{frcm}\cite{AR}
\em{ If the facet ideal $I_{\mathcal{F}}(\Delta)$ of a pure
simplicial complex $\Delta$ over $[n]$ has quasi-linear quotients,
then the face ring is Cohen Macaulay.}
\end{cor}

%\begin{Proposition}
%Let $S=K[x_1,x_2,\hdots,x_m]$ be a polynomial ring over field and $I_1=\langle x_{s_1},x_{s_2},\hdots,x_{s_u}\rangle$ and $I_2=\langle x_{t_1},x_{t_2},\hdots,x_{t_v}\rangle$ be monomial ideals in $S$ such that $G(I_1)\bigcap G(I_2)=\{x_{s_1},x_{s_2},\hdots,x_{s_\alpha}\}$ then $I_1:I_2=\langle x_{s_{\alpha+1}},x_{s_{\alpha+2}},\hdots,x_{s_u}\rangle. $
%\end{Proposition}
%\begin{proof}
%By definition of colon ideal we have
%$$I_1:I_2=(\frac{x_{s_1},x_{s_2},\hdots,x_{s_u}}{gcd((x_{s_1},x_{s_2},\hdots,x_{s_u}),(x_{t_1},x_{t_2},\hdots,x_{t_v}))})$$
%Since $gcd((x_{s_1},x_{s_2},\hdots,x_{s_u}),(x_{t_1},x_{t_2},\hdots,x_{t_v}))=G(I_1)\bigcap G(I_2)$,\\
%Therefore, $I_1:I_2=\langle x_{s_{\alpha+1}},x_{s_{\alpha+2}},\hdots,x_{s_u}\rangle. $
%\end{proof}

\begin{Theorem}\label{CM}
\em{The face ring of $\Delta_s(\mathcal{G}_{n,r}^1)$ is
Cohen-Macaulay.}
\end{Theorem}
\begin{proof}
By corollary \ref{frcm}, it is sufficient to show that $I_{\mathcal{F}}\big(\Delta_s(\mathcal{G}_{n,r}^1)\big)$ has a quasi-linear quotients in $S=k[x_{11},x_{12},\hdots,x_{1m_1},x_{21},x_{22},\hdots,x_{2(m_2-1)},\hdots,x_{r1},x_{r2},\hdots,\\
x_{r(m_r-1)},x_1,x_2,\hdots,x_t]$. By lemma \ref{scn}, we have
$$ s (\mathcal{G}_{n,r}^1)=\mathcal{C}_{(1)}\bigcup\mathcal{C}_{(2)}\bigcup\mathcal{C}_{(3a)}\bigcup\mathcal{C}_{(3b)}\bigcup\mathcal{C}_{(3c)}$$
Therefore,
$$ \Delta_s (\mathcal{G}_{n,r}^1)=\Big\langle\hat{E}_{(j_1i_1,j_2i_2,\hdots,j_ri_r)}=E\backslash \{e_{j_1i_1},e_{j_2i_2},\hdots,e_{j_ri_r}\}\mid \hat{E}_{(j_1i_1,j_2i_2,\hdots,j_ri_r)}\in s (\mathcal{G}_{n,r}^1)\Big\rangle$$
and hence we can write,
$$ I_{\mathcal{F}}(\Delta_s(\mathcal{G}_{n,r}^1))=\Big( x_{\hat{E}_{(j_1i_1,j_2i_2,\hdots,j_ri_r)}}\mid \hat{E}_{(j_1i_1,j_2i_2,\hdots,j_ri_r)}\in s (\mathcal{G}_{n,r}^1)\Big).$$
Here $ I_{\mathcal{F}}(\Delta_s(\mathcal{G}_{n,r}^1))$ is a pure
monomial ideal of degree $n-r$ with
$x_{\hat{E}_{(j_1i_1,j_2i_2,\hdots,j_ri_r)}}$ as the product of all
variables in $S$ except $x_{j_1i_1},x_{j_2i_2},\hdots,x_{j_ri_r}$.
Now we will show that $
I_{\mathcal{F}}(\Delta_s(\mathcal{G}_{n,r}^1))$ has quasi-linear
quotients with respect to the following generating system:\\
$\{x_{\hat{E}_{(11,21,\hdots,r1)}}\},\{x_{\hat{E}_{(11,21,\hdots,(r-1)1,j_ri_r)}}\mid i_r\neq 1\},\{x_{\hat{E}_{(11,21,\hdots,(r-2)1,(r-1)i_{r-1},j_ri_r)}}\mid i_{r-1}\neq 1\},\\
\{x_{\hat{E}_{(11,21,\hdots,(r-3)1,(r-2)i_{r-2},j_{r-1}i_{r-1},j_ri_r)}}\mid i_{r-2}\neq 1\},\hdots ,\{x_{\hat{E}_{(11,2i_2,j_3i_3\hdots,j_ri_r)}}\mid i_2\neq 1\},\\
\{x_{\hat{E}_{(1i_1,j_2i_2,\hdots,j_ri_r)}}\mid i_1\neq 1\}$\\
Let us put
\\$\begin{array}{c}
  C_{(11,21,\hdots,(r-1)1,j_ri_r)}=\{x_{\hat{E}_{(11,21,\hdots,(r-1)1,j_ri_r)}}\mid i_r\neq 1\}, \\
  C_{(11,21,\hdots,(r-2)1,(r-1)i_{r-1},j_ri_r)}=\{x_{\hat{E}_{(11,21,\hdots,(r-2)1,(r-1)i_{r-1},j_ri_r)}}\mid i_{r-1}\neq 1\}, \\
  \vdots \\
   C_{(1i_1,j_2i_2,\hdots,j_ri_r)}=\{x_{\hat{E}_{(1i_1,j_2i_2,\hdots,j_ri_r)}}\mid i_1\neq 1\}.
\end{array}$\\
%$C_{(11,21,\hdots,(r-1)1,j_ri_r)}=\{x_{\hat{E}_{(11,21,\hdots,(r-1)1,j_ri_r)}}\mid i_r\neq 1\}$
 Also for any $C_{(j_1i_1,j_2i_2,\hdots,j_ri_r)}$, denote $\bar{C}_{(j_1i_1,j_2i_2,\hdots,j_ri_r)}$  as the residue collection of all the generators which precedes $C_{(j_1i_1,j_2i_2,\hdots,j_ri_r)}$ in the above order. We will show that
$$(\bar{C}_{(j_1i_1,j_2i_2,\hdots,j_ri_r)}):( x_{\hat{E}_{(j_1i_1,j_2i_2,\hdots,j_ri_r)}})$$
contains atleast one linear generator.\\
Now for any generator $x_{\hat{E}_{(11,\hdots,(k-1)1,j_ki_k,\hdots,j_ri_r)}}$, the above said system of generators guarantee the existence of a generator $x_{\hat{E}_{(11,\hdots,(k-1)1,j_\alpha i_\alpha,j_{k+1}i_{k+1},\hdots,j_ri_r)}}$ in $\bar{C}_{(11,\hdots,(k-1)1,j_ki_k,\hdots,j_ri_r)}$ such that $j_\alpha i_\alpha\neq j_ki_k$. Therefore, by using the definition of colon ideal it is easy to see that
$$(\bar{C}_{(11,\hdots,(k-1)1,j_ki_k,\hdots,j_ri_r)}):(x_{(11,\hdots,(k-1)1,j_ki_k,\hdots,j_ri_r)})$$
contains a linear generator $x_{j_ki_k}$. Hence $I_{\mathcal{F}}(\Delta_s(\mathcal{G}_{n,r}^1))$ has quasi-linear quotients, as required.
%is guaranteed by the above order to contain a generator which does not contain the product $x_{j_1i_1},\hdots,x_{j_{k-1}1},x_{j_{k+1}i_{k+1}},\hdots,x_{j_ri_r}$
\end{proof}
We conclude this section with an example.
\begin{Example}
\em{For the graph $\mathcal{G}_{10,2}^1$ given in Fig. 1.,
%\begin{center}
%\begin{picture}(300,80)\label{fig}
%\thicklines \put(100,27){\line(1,0){69.6}}\put(100,60){\line(1,0){70}}
%\put(96,25){$\bullet$}
%\put(77,27){\line(1,0){30}}\put(75,24){$\bullet$}\put(85,18){${e_{13}}$}\put(104,40){${e_{11}}$}
%
%\put(57,27){\line(1,0){30}}\put(77,27){\line(2,3){23}}\put(55,24){$\bullet$}\put(65,18){${e_{1}}$}
%\put(99,27){\line(0,1){33}}\put(97,57){$\bullet$}
%\put(70,40){$e_{12}$}\put(172,40){$e_{21}$}\put(200,40){$e_{3}$}
%\put(132,18){${e_{23}}$}\put(182,18){${e_{2}}$}\put(205,18){${e_{4}}$}
%\put(132,64){$e_{22}$}
%\put(165.5,25){$\bullet$}
%\put(166,27){\line(1,0){33}}\put(195,24.5){$\bullet$}
%\put(197,27){\line(1,0){33}}\put(225,24.5){$\bullet$}
%\put(197,27){\line(0,1){30}}\put(195,54){$\bullet$}
%
%\put(168,27){\line(0,1){33}}\put(165.5,57){$\bullet$}
%\put(116,-4){\tiny Fig. 1 . $\mathcal{G}_{10,2}$}
%\end{picture}
%\end{center}}
the facet ideal of the spanning simplicial complex is:\\
$$I_{\mathcal{F}}(\Delta_s(\mathcal{G}_{10,2}^1))=(x_{11,21},C_{11,j_2i_2},C_{1i_1,j_2i_2})$$
where $C_{11,j_2i_2}=x_{11,22},x_{11,23},x_{11,12},x_{11,13}$ and $C_{1i_1,j_2i_2}=x_{12,21},x_{12,22},x_{12,23},x_{13,21},x_{13,22},x_{13,23}.$
It is easy to see that, $I_{\mathcal{F}}(\Delta_s(\mathcal{G}_{10,2}^1))$ has quasi-linear quotients with respect to the ordering given to its generators (by applying above theorem). Hence the face ring of $\Delta_s(\mathcal{G}_{10,2}^1)$ is Cohen Macaulay.}

\end{Example}

\end{document}